\documentclass[12pt]{iopart}

\usepackage{iopams}  

\usepackage{amsgen, amsfonts, amsbsy, amssymb}

\usepackage{amsthm}
\theoremstyle{plain}
\newtheorem*{theorem-M}{Main Theorem}
\newtheorem*{corollary-B}{Corollary for Coercive Potentials}
\newtheorem*{theorem-C}{Collateral Theorem}
\newtheorem*{theorem-E}{Existence Theorem}
\newtheorem*{theorem-F}{Finite Cyclic Predecessor Version}

\newtheorem{theorem}{Theorem}[section]
\newtheorem{lemma}[theorem]{Lemma}
\newtheorem{proposition}[theorem]{Proposition}
\newtheorem{corollary}[theorem]{Corollary}

\theoremstyle{definition}
\newtheorem{definition}[theorem]{Definition}

\theoremstyle{remark}
\newtheorem{remark}[theorem]{Remark}

\newcommand{\stacksub}[1]{
\mbox{\scriptsize 
$\begin{array}{c}
#1
\end{array}$}
}

\usepackage[dvipsnames]{xcolor}

\usepackage{epsfig}
\usepackage{graphicx}
\usepackage{tikz}
\usetikzlibrary{arrows, automata, backgrounds, chains, matrix, shadows, positioning, calc, scopes, fadings, shapes.geometric, patterns}

\usepackage{enumitem}

\usepackage{dsfont}

\usepackage{hyperref}
\hypersetup{colorlinks,  citecolor=black,  filecolor=black,  linkcolor=black,  urlcolor=black}

\begin{document}

\title[Dense periodic optimization for CMS via Aubry points]{Dense periodic optimization for countable Markov shift via Aubry points}

\author{Eduardo Garibaldi$^1$ \& Jo\~ao T A Gomes$^2$}

\address{
$^1$ Institute of Mathematics, Statistics and Scientific Computing,
University of Campinas,
13083-859, Campinas, SP, Brazil \\
$^2$ Center of Exact and Technological Sciences, 
Federal University of Rec\^oncavo da Bahia,
44380-000, Cruz das Almas, BA, Brazil
}

\eads{
\mailto{garibaldi@ime.unicamp.br} 
\enspace and \enspace 
\mailto{jtagomes@ufrb.edu.br}
}

\vspace{10pt}
\begin{indented}
\item[] February 2026
\end{indented}

\begin{abstract} \\
For transitive Markov subshifts over countable alphabets, this note ensures that a dense subclass of locally H\"older continuous potentials admits at most a single periodic probability as a maximizing measure.
We resort to concepts analogous to those introduced by Mather and Ma\~n\'e in the study of globally minimizing curves in Lagrangian dynamics.
In particular, given a summable variation potential, we show the existence of a continuous sub-action in the presence of an Aubry point.

\vspace{0.6pc}
\noindent{\it Keywords\,}:
Aubry Set,
Countable Markov Shift,
Ergodic Optimization,
Ma\~n\'e Potential,
Sub-actions,
Peierls Barrier.

\vspace{0.6pc}
\noindent{\it Mathematics Subject Classification number\/}:
28D05, 
37A05, 
37B10. 

\end{abstract}

\hypertarget{sec--introduction}{}
\section{Introduction and Main Results}
\label{section-introduction}

Ergodic optimization is the study of problems relating to maximizing orbits and invariant measures, and maximum ergodic averages. 
The field emerged in the 1990s as a natural extension of classical ergodic theory, seeking to identify invariant measures or orbits that maximize a given potential function. 
A comprehensive survey of the field is provided by Jenkinson~\cite{Jen:ETDS19}, which reviews the theoretical foundations and typical properties of maximizing measures in topological dynamical systems.

The development of ergodic optimization theory has been particularly rich in the context of expanding dynamical systems, where it found deep connections with classical variational principles. 
A pivotal moment in this development was the recognition that key concepts from Aubry-Mather theory for Lagrangian systems, originally formulated by Ma\~n\'e \cite{Man:BSBM97, CDI:BSBM97}, could be adapted to the discrete setting of symbolic dynamics. This adaptation involved translating fundamental tools such as sub-actions -- which serve as analogues to subsolutions of the Hamilton-Jacobi equation -- into the symbolic context. 
The work of Contreras, Lopes, and Thieullen \cite{CLT:ETDS01} represented a landmark achievement in this program, establishing a coherent theoretical framework that revealed previously unsuspected parallels between continuous Lagrangian dynamics and discrete optimization problems. Their contributions included foundational results on calibrated and separating sub-actions, which became essential tools for understanding the structure of maximizing measures in symbolic dynamics. However, this adaptation was not without subtleties: while many tools from Aubry-Mather theory found natural analogues in ergodic optimization, the correspondence is not always direct, and certain fundamental features -- such as the central role of convexity in Aubry-Mather theory -- do not appear in the same form in the general ergodic optimization setting. This synthesis of ideas from classical mechanics and ergodic theory proved to be highly influential in the subsequent development of ergodic optimization (see \cite{Gar:SPR17}). The theoretical framework established through these contributions helped set the stage for addressing increasingly sophisticated questions about the generic behavior of maximizing measures, leading to the groundbreaking results we now celebrate.

We acknowledge with deep respect the late Professor Gonzalo Contreras, whose contributions profoundly shaped the field of ergodic optimization. Beyond his co-authorship of the influential monograph ``Global minimizers of autonomous Lagrangians'' in Aubry-Mather theory, reference \cite{CI:IMPA99}, Contreras made decisive contributions to understanding maximizing measures in expanding systems through his collaboration with Lopes and Thieullen. Most significantly, in \cite{Con:IM16}, Contreras ultimately resolved Ma\~n\'e conjecture for the compact expanding case, demonstrating that for an expanding map on a compact metric space, a generic Lipschitz potential has a unique maximizing measure supported on a single periodic orbit. This theorem, published in \textit{Inventiones Mathematicae} in 2016, resolved a fundamental question in ergodic optimization that had amoung its origins the typically periodic optimization conjecture proposed by Ma\~n\'e in the 1990s. Throughout his distinguished career at CIMAT in Guanajuato, Mexico, Contreras left an indelible mark on the mathematical community, and his passing in 2025 represents a significant loss to the field.

Our main result extends Contreras's framework to non-compact settings. While his periodic-orbit theorem characterized the generic behavior of maximizing measures in compact expanding systems, we investigate how these principles translate to symbolic dynamical systems with countable alphabets. In this context, the absence of compactness introduces new technical challenges that require careful treatment of convergence properties and the behavior of action potentials. By adapting the techniques developed by Contreras and his collaborators to the countable Markov shift setting, we demonstrate that the fundamental insights about periodic maximizers can be preserved even when the classical compactness assumptions are relaxed.

Specifically, ergodic optimization in the context of countable Markov shifts has been mainly focused on ensuring existence and describing properties of optimizing measures \cite{JMU:ETDS06, BG:BBMS10,BF:ETDS13}. 
Through a perturbative approach, we show here the denseness within a class of locally H\"older continuous potentials of the subset formed by those that admit at most one periodic probability as maximizing measure. 
To the best of our knowledge, this is an inaugural result of typicality for optimal measures in this scenario where the compactness of the phase space is not assumed.
Our strategy follows the one adopted for the case of finite alphabets \cite{Con:IM16,HLMXZ:JEMS25,Boc:Notes19}. 
In addition to the main result, another contribution of these notes to ergodic optimization over non-compact phase spaces is the adaptation to a symbolic dynamic setting of the viewpoint developed by Ma\~n\'e for Lagrangian systems \cite{CI:IMPA99}. 
The concepts discussed here had not yet been considered with the depth that we were led to take into account. 
In particular, we had to return to the very concept of the ergodic maximizing constant via Ma\~n\'e's critical value, an attitude that, to our knowledge, had not yet been used in ergodic optimization. 

Let $ \Sigma $ denote a one-sided Markov shift on a countable alphabet, that is, the set of one-sided infinite sequences $ (x_0, x_1, x_2, \ldots) $ where each $ x_i $ belongs to a countable alphabet and consecutive symbols satisfy prescribed transition rules given by an zero-one adjacency matrix, and let $ \sigma \,:\, \Sigma \to \Sigma $ denote the left shift map defined by $ \sigma((x_0, x_1, x_2, \ldots)) = (x_1, x_2, x_3, \ldots) $.
To be concrete, we assume from this point forward that the alphabet is the set of non-negative integers $ \mathbb{Z}_{+} $.
As usual, $ \mathbb{Z}_{+} $ is provided with the discrete topology and $ \Sigma $ with the product topology, the latter being metrizable when, for a fixed $ \lambda \in (0,1) $, the distance between two sequences $ x = (x_0, x_1, \ldots) $ and $ y = (y_0, y_1, \ldots) $ is $ d (x, y) = \lambda^\ell $, where $ \ell $ indicates the first position of disagreement.
We assume throughout the text that the dynamics $ ( \Sigma, \sigma ) $ is topologically transitive (i.e., the orbit of any open set eventually intersects any other open set). For comprehensive treatments of the theory of Markov shifts on countable alphabets, we refer to the monographs of Kitchens \cite{Kit:SPR98} and Lind--Marcus \cite{LM:CUP95}.

We call a potential any continuous function $ A \,:\, \Sigma \to \mathbb{R} $ bounded from above. 
The main objective is to study (when they exist) the $ \sigma $-invariant probabilities that maximize the integral of $ A $ over $ \Sigma $.
The regularity of the potential plays an important role in this analysis. 
In fact, much of the subtlety of results in ergodic optimization lies in this specific aspect.
For subshifts over finite alphabets, for instance, one can either generically have a single maximizing probability supported on a periodic orbit by focusing on H\"older continuous potentials \cite{Con:IM16}, or observe in a dense way the phenomenon of the existence of uncountably many ergodic maximizing measures with full support and positive entropy when taking into account continuous potentials in general \cite{Shi:N18}.

We first consider potentials with summable variation.
For a function $ A : \Sigma \to \mathbb R $ and a non-negative integer $ \ell \geq 0 $, the $ \ell $-th variation is defined as
\[
\mathrm{Var}_{\ell} (A) 
= \sup \left\{ A(x) - A(y) \, : \,  d(x,y) \leq \lambda^\ell \right\}.
\]
For all $ 0 \leq n < m \leq \infty $, we denote $ \mathrm{Var}_{n}^{m} (A) := \sum_{\ell = n}^{m} \mathrm{Var}_{\ell} (A) $. 
We say that $ A $ is of summable variation when
\[
\mathrm{Var}_{1}^{ \infty } (A) := \sum_{\ell = 1}^{\infty} \mathrm{Var}_{\ell} (A) < + \infty.
\]
(Note that no restriction is imposed on the zeroth variation of $A$.)
The space of real-valued functions of summable variation is a disjoint collection of affine spaces, 
each of which is a metric space with respect to the distance
\[
\| A - B \|_{\mathrm{sv}} := \mathrm{Var}_{1}^{ \infty } (A - B) + \| A - B \|_\infty,
\]
where $ \| \cdot \|_\infty $ denotes the supremum norm. 
A specific class will receive our attention: a potential $ A $ is said to be locally H\"older continuous when there exists a positive constant $ \mathrm{Lip}_{\mathrm{loc}}(A) $, called local H\"older constant, such that 
\[
\mathrm{Var}_{\ell} (A) 
\leq \mathrm{Lip}_{\mathrm{loc}}(A) \, \lambda^\ell 
\qquad \mbox{for all} \qquad
\ell \geq 1.
\]

Given a potential $ A $, we introduce the ergodic maximizing constant as
\[
\beta_A = \sup \left\{ \int_{\Sigma} \, A \, \rmd \mu \,:\, \mu \mbox{ is a $ \sigma $-invariant probability} \right\}.
\]
Obviously $ \beta_A \in \big(-\infty, \, \sup A \big] $. 
We say that an invariant probability measure $ \mu $ is maximizing whenever $ \int A \, \rmd \mu = \beta_A $.
The existence of maximizing measures is far from obvious in the general context of Markov shifts.
When these exist, however, they are expected to be supported in the Aubry set, the habitat of maximizing trajectories.
In a concise formulation, the Aubry set captures the recurrent points whose finite orbit segments consistently achieve optimality in their long-term sum behavior.
Its precise definition is as follows
(for details, see Chapter~4 of~\cite{Gar:SPR17}). 

\begin{definition}[Aubry set]  \label{definition-aubry-set}
We say that $ x \in \Sigma $ is an Aubry point for the potential $ A $ when, for any $ \varepsilon > 0 $, there are a point $ w \in \Sigma $ and an integer $ n > 0 $ such that 
$ d \left( x, w \right) < \varepsilon $,
$ d \left( \sigma^n (w), x \right) < \varepsilon $ and 
\[
- \varepsilon < S_n (\beta_A - A) (w) < \varepsilon,
\]
where $ S_n $ indicates the $n$-th Birkhoff sum. 
The set of Aubry points is denoted by $ \Omega(A) $.
\end{definition}

Our main result can be stated as follows. 
 
\begin{theorem-M}  \hypertarget{MThm}{}
Let $ (\Sigma, \sigma) $ be a topologically transitive Markov shift on a countable alphabet. 
Let $ A $ be a locally H\"older continuous potential. 
Suppose that its Aubry set contains a non-empty compact invariant subset. 
Then, for any $ \varepsilon > 0 $, there exists a locally H\"older continuous potential $ B $ such that $ \| A - B \|_{\mathrm{sv}} < \varepsilon $ and $ B $ has at most a single periodic probability as maximizing measure.
\end{theorem-M}

The formulation ``at most a single periodic probability'' in our~\hyperlink{MThm}{Main Theorem} reflects a fundamental distinction between our setting and the classical finite alphabet case. In the general context of countable Markov shifts, the existence of maximizing measures is not guaranteed a priori. Our perturbative approach ensures that the perturbed potential $ B $ has at most one periodic maximizing probability, but does not necessarily guarantee its existence. However, this technical difficulty in ensuring existence should not be interpreted as indicating that periodic maximization is a rare phenomenon in our setting  -- quite the contrary, as demonstrated by our~\hyperlink{cor-B}{Corollary for Coercive Potentials} (stated below), where periodic optimization occurs densely. Moreover, when the underlying shift satisfies an additional structural condition -- roughly speaking, that the shift space exhibits sufficiently well-controlled periodic behavior -- our techniques can be strengthened to produce a perturbed potential $ B $ that admits exactly one maximizing measure supported on a periodic orbit. This refinement, which yields the more natural statement ``exactly a single periodic probability,'' builds upon a recently obtained result \cite{GGM:ETDS26} and is detailed in the~\hyperlink{app-A}{Appendix A}.

The compactness assumption in the~\hyperlink{MThm}{Main Theorem}, while necessary for our perturbative approach, does not reduce the problem to the previously solved finite alphabet cases treated in \cite{Con:IM16} and \cite{HLMXZ:JEMS25}. The essential distinction lies in the nature of the underlying phase space: even when maximizing measures have compact support, the ambient space $ \Sigma $ remains non-compact, and the potential $ A $ is defined on this larger space. This creates fundamental technical challenges absent in the finite alphabet setting, particularly regarding the extension of optimization arguments to the full non-compact space -- including the convergence properties of action potentials and the behavior of sub-actions beyond the support of maximizing measures, concepts that will be developed in the sections that follow. The requirement that ``the Aubry set contains a non-empty compact invariant subset'' should be understood as a structural condition ensuring that optimization occurs within a tractable region, rather than a reduction to finite alphabet dynamics. Note that, thanks to Atkinson's theorem \cite{Atk:JLMS76}, this hypothesis is equivalent to the existence of a maximizing probability with compact support. Examples of potentials that fulfill this requirement are provided by the coercive ones \cite[Theorem~1]{BF:ETDS13}. Recall that a potential $ A $ is said to be coercive when 
\[
 \lim_{i \to +\infty} \sup A|_{[i]} = - \infty, 
\] 
where $ [i] $ is the cylinder set $ \{ x = (x_0, x_1, \ldots) \in \Sigma : x_0 = i \} $. The significance of this class becomes apparent when we observe that coercive potentials naturally force optimization to occur within bounded regions of the symbol space, while still being defined on the full non-compact shift space, thus illustrating how the compactness condition serves as a structural requirement for tractability rather than an artificial restriction. Notably, coercive potentials are inherently tied to the countable alphabet setting and have no natural analogue in the finite alphabet context previously studied. Thus, a particular consequence of our main result is the following one.

\begin{corollary-B} \hypertarget{cor-B}{}
Let $ (\Sigma, \sigma) $ be a topologically transitive Markov shift on a countable alphabet. 
Let $ A $ be a coercive locally H\"older continuous potential. 
Then, given $ \varepsilon > 0 $, there exists a coercive locally H\"older continuous potential $ B $ such that $ \| A - B \|_{\mathrm{sv}} < \varepsilon $ and $ B $ admits a single periodic probability as maximizing measure.
\end{corollary-B}

For primitive subshifts, the so-called oscillation condition (for details, see \cite[Definition~5.1]{JMU:ETDS06}) actually allows one to consider a more general class of examples \cite[Theorem~6.1]{JMU:ETDS06}. 
As a matter of fact, the central results in \cite{JMU:ETDS06, BG:BBMS10,BF:ETDS13} ensure that, in the cases analyzed, any maximizing probability is indeed supported in a subshift over a finite subalphabet.
Our main theorem takes into account also other situations.
For instance, if $ \Sigma_0, \Sigma_1 \subset \Sigma $ are disjoint subshifts, the first one over a finite subalphabet and the second one over a countable subalphabet, the result applies to the potential $ A = - d(\Sigma_0, \cdot) d(\Sigma_1, \cdot) $, whose maximizing probabilities are clearly all those invariant ones supported either in $ \Sigma_0 $ or in $ \Sigma_1 $.

A key ingredient for obtaining the above theorem is the notion of a sub-action, namely, 
a continuous function $ u : \Sigma \to \mathbb R $ such that 
\[
A + u \circ \sigma - u \le \beta_A 
\qquad \mbox{ everywhere on } \quad \Sigma.
\]
Previously, the notion of normal form \cite[Definition 2.2]{JMU:ETDS06} consisted of an interesting proposal to extend the role played by sub-actions in characterizing maximizing probabilities.
An ancillary result of this work is the guarantee of existence of a sub-action via the Ma\~n\'e potential $ \phi_A $ or the Peierls barrier $ h_A $, both defined on $ \Sigma \times \Sigma $, objects hitherto unexplored in the general context of potentials of summable variation on transitive countable Markov shifts.
Their precise definition is postponed until the \hyperlink{sec-mane-potential-peierls-barrier}{next section}. 
In the Lagrangian setting, the pioneer notions were introduced in \cite{Mat:AIF93, Man:BSBM97}.

\begin{theorem-C} \hypertarget{CThm}{}
Let $ (\Sigma, \sigma) $ be a topologically transitive Markov shift on a countable alphabet. 
Let $ A $ be a potential of summable variation. 
Suppose there exists an Aubry point for $ A $.
Then, for any $ x \in \Omega (A) $ fixed, the potential admits a sub-action defined as
\[
u_{x} ( \,\cdot\, ) = \phi_A ( x, \,\cdot\, ) = h_A( x, \,\cdot\, ),
\]
which has  $\ell$-th variation bounded from above by $ 2 \, \mathrm{Var}_{\ell}^\infty (A) $. 
If $
\sum_{\ell = 1}^\infty \mathrm{Var}_{\ell}^\infty (A) < + \infty $, the sub-action
$ u_x $ has summable variation. 
In particular, a locally H\"older continuous potential admits a locally H\"older continuous sub-action.
\end{theorem-C}

The paper is organized as follows. 
In Section~\ref{section-mane-potential-peierls-barrier}, we introduce both the Ma\~n\'e potential and the Peierls barrier and highlight their main properties.
The \hyperlink{sec-aubry-set}{third section} is devoted to the Aubry set, with special attention to its interactions with these action potentials. 
In particular, \hyperlink{CThm}{Collateral Theorem} corresponds to Theorem~\ref{theorem-sub-action}.
The proof of \hyperlink{MThm}{Main Theorem} is presented in the \hyperlink{sec-densely-periodic-optimization}{final section} of this note. \hyperlink{app-A}{Appendix A} presents the strengthened version of our main result under additional structural assumptions on the shift space, where the perturbation yields exactly one periodic maximizing measure.

\hypertarget{sec-mane-potential-peierls-barrier}{}
\section{Ma\~n\'e Potential and Peierls Barrier}
\label{section-mane-potential-peierls-barrier}

\subsection{Fundamental Facts}
\label{subsection-fundamental-fact-mane-peierls}

Both the Ma\~n\'e potential and the Peierls barrier are action potentials between points, the first considers trajectories of any size, the second focuses on arbitrarily long trajectories. 
We can introduce them through the following auxiliary function.

\begin{definition}   \label{definition-auxiliary-function}
Let $ A : \Sigma \to \mathbb{R} $ be a potential and $ \gamma \in \mathbb{R} $ be a constant.
Given integers $ k \ge 0 $ and $ l \geq 0 $, we define for $ x, y \in \Sigma $,
\[
{}_{A}^{\gamma}\mathfrak{S}_{l}^{k} (x,y) 
:= \inf_{n \geq l} 
\inf_{ \stacksub{ d(x,w) \leq \lambda^k \\ d(\sigma^n(w),y) \leq \lambda^k } }
S_n ( \gamma - A )(w).
\]
\end{definition}

To avoid cumbersome notation, when it is clear the potential taken into account, we will simply denote $ {}^{\gamma}\mathfrak{S}_{l}^{k} (x, y) $. 
Likewise, when $ \gamma = \beta_A $, we will just use $ \mathfrak{S}_{l}^{k} (x,y) $.

Concerning its basic properties, this function clearly fulfills, for all $ x, y \in \Sigma $,
\begin{eqnarray}
0 \leq k, l 
\quad & \Longrightarrow \quad
& {}^{\gamma}\mathfrak{S}_{l}^{k} (x,y) < + \infty;
\label{equation-auxiliary-non-plus-infinity} \\
0 \leq l_1 \leq l_2 
\quad & \Longrightarrow \quad
& {}^{\gamma}\mathfrak{S}_{l_1}^{k} (x,y) \leq {}^{\gamma}\mathfrak{S}_{l_2}^{k} (x,y);
\label{equation-auxiliary-monotonicity-l} \\
0 \le k_1 \leq k_2
\quad & \Longrightarrow \quad
& {}^{\gamma}\mathfrak{S}_{l}^{k_1} (x,y) \leq {}^{\gamma}\mathfrak{S}_{l}^{k_2} (x,y).
\label{equation-auxiliary-monotonicity-k}
\end{eqnarray}
Besides, the auxiliary function is locally constant, i.e.,
\begin{equation}
{}^{\gamma}\mathfrak{S}_{l}^{k} (x, y) = {}^{\gamma}\mathfrak{S}_{l}^{k} (x', y') 
\quad \mbox{whenever} \quad
d(x,x') \leq \lambda^k \mbox{ and } d(y,y') \leq \lambda^k.
\label{equation-locally-constant-function}
\end{equation}

Note that even if the infimum in the definition of auxiliary function is not $ + \infty $, the above result does not prevent the supremum with respect to $ k $ of $ {}^{\gamma}\mathfrak{S}_{l}^{k} $  to be $ + \infty $.
This fact will lead us to pay close attention to situations in which $ \pm \infty $ values can be present.

A fundamental inequality involving the auxiliary function is the following one.

\begin{lemma}  \label{lemma-auxiliary-inequality}
Let $ A : \Sigma \to \mathbb{R} $ be a potential of summable variation and $ \gamma \in \mathbb{R} $ be a constant.
For all integers $ k \geq \bar{k} > 0 $ and $ l \geq k - \bar{k} $, $ m \geq 0 $ and for any points $ x, y, z \in \Sigma $, we have
\[
{}^{\gamma}\mathfrak{S}_{l+m}^{ k } (x,z) 
\leq {}^{\gamma}\mathfrak{S}_{l}^{ k } (x,y) 
+ {}^{\gamma}\mathfrak{S}_{m}^{ \bar{k} } (y,z) 
+ 2 \, \mathrm{Var}_{ \bar{k} }^{\infty} (A).
\]
\end{lemma}

This result corresponds to a version of Lemma~5.1 of \cite{Gar:SPR17} for countable alphabets and potentials of summable variation.
\begin{proof}
Since $ {}^{\gamma}\mathfrak{S}_{l}^{ k } (x,y) $ and $ {}^{\gamma}\mathfrak{S}_{m}^{ \bar{k} } (y,z) $ are infimums and due to property~\eref{equation-auxiliary-non-plus-infinity}, given $ \eta > 0 $, there exist $ w $, $ \bar{w} \in \Sigma $ and integers $ n \geq l $, $ \bar{n} \geq m $ such that
\begin{eqnarray} 
\fl
d(x, w) \leq \lambda^k, \qquad
d(\sigma^n(w), y) \leq \lambda^k
\qquad 
& \mbox{and} \qquad
S_n( \gamma - A )(w) < {}^{\gamma}\mathfrak{S}_{l}^{ k } (x,y) + \eta;
\label{equation-sn-auxiliary} \\
\fl
d( y, \bar{w}) \leq \lambda^{ \bar{k} }, \qquad
d(\sigma^{ \bar{n} }(\bar{w}), z) \leq \lambda^{ \bar{k} }
\qquad 
& \mbox{and} \qquad
S_{ \bar{n} }( \gamma - A )( \bar{w} ) < {}^{\gamma}\mathfrak{S}_{m}^{ \bar{k} } (y,z) + \eta.
\label{equation-sbarn-auxiliary}
\end{eqnarray}
(In the case where 
$ {}^{\gamma}\mathfrak{S}_{l}^{ k } (x,y) $ or 
$ {}^{\gamma}\mathfrak{S}_{m}^{ \bar{k} } (y,z)$ 
equals $ - \infty $, we consider $ M > 0 $, sufficiently large, such that 
$ S_{ n }( \gamma - A )( w ) < - M $ or 
$ S_{ \bar{n} }( \gamma - A )( \bar{w} ) < - M $.)

From the above distances, we obtain
$ \big( w_0, \ldots, w_{ k - 1 } \big) = \big( x_0, \ldots, x_{ k - 1 } \big) $,
$ \big( w_n, \ldots, w_{ n + \bar{k} - 1 } \big) 
= \big( y_0, \ldots, y_{ \bar{k} - 1 } \big)
= \big( \bar{w}_0, \ldots, \bar{w}_{ \bar{k} - 1 } \big) $
and
$ \big( \bar{w}_{ \bar{n} }, \ldots, \bar{w}_{ \bar{n} + \bar{k} - 1 } \big) = \big( z_0, \ldots, z_{ \bar{k} - 1 } \big) $.
Since $ w_n = y_0 = \bar{w}_{0} $, 
we observe that the sequence
\[
\hat{w} 
= \big( w_{0}, w_{1},\ldots, w_{n-1}, \,
\bar{w}_{0}, \bar{w}_{1}, \ldots, \bar{w}_{\bar{n} - 1}, \,
z_0, z_1, \ldots \big).
\]
belongs to $ \Sigma $.
By construction,
$ d(w, \hat{w} ) \leq \lambda^{ n + \bar{k} } $
and
$ d( \sigma^{ n } ( \hat{w} ), \bar{w} )
\leq \lambda^{ \bar{ n } + \bar{k} } $,
so we have the bounds
\begin{eqnarray}   
S_{ n }( \gamma - A )( \hat{w} )
\leq S_{ n }( \gamma - A )( w )
+ \mathrm{Var}_{ \bar{k} }^{\infty} (A);
\label{equation-sn-variation} \\
S_{ \bar{n} }( \gamma - A ) \circ \sigma^n ( \hat{w} )
\leq S_{ \bar{n} }( \gamma - A )( \bar{w} )
+ \mathrm{Var}_{ \bar{k} }^{\infty} (A).
\label{equation-sbarn-variation}
\end{eqnarray}

Analyzing the relative positions of $ n $ and $k $, we obtain 
$ d \left(  x, \hat{w} \right) \leq \lambda^{ k \, \wedge \, ( n + \bar{k} ) } $.
Since $ n \geq l \geq k -\bar{k} $, we have
$ d \left( x, \hat{w} \right) \leq \lambda^{ k } $.
Moreover, $ n + \bar{n} \geq l + m $ and
$ \sigma^{ n + \bar{n} } ( \hat{w} ) = z $.
We conclude that
\[
\begin{array}{ r @{\,} c @{\,} l }
{}^{\gamma}\mathfrak{S}_{ l + m }^{ k } (x,z)
& \leq & S_{ n + \bar{n} }( \gamma - A )( \hat{w} ) 
= S_{ n }( \gamma - A )( \hat{w} )
+ S_{ \bar{n} }( \gamma - A ) \circ \sigma^n ( \hat{w} )
\\
& \stackrel{ \eref{equation-sn-variation} \, \eref{equation-sbarn-variation} }{ \leq } 
& S_{ n }( \gamma - A )( w ) 
+ \mathrm{Var}_{ \bar{k} }^{\infty} (A) 
+ S_{ \bar{n} }( \gamma - A )( \bar{w} ) 
+ \mathrm{Var}_{ \bar{k} }^{\infty} (A)
\\
& \stackrel{ \eref{equation-sn-auxiliary} \, \eref{equation-sbarn-auxiliary} }{ < } 
& {}^{\gamma}\mathfrak{S}_{l}^{ k } (x,y) + \eta + {}^{\gamma}\mathfrak{S}_{m}^{ \bar{k} } (y,z) + \eta 
+ 2 \, \mathrm{Var}_{ \bar{k} }^{\infty} (A).
\end{array}
\]
The result follows by taking $ \eta $ arbitrarily small.
\end{proof}

We initially consider versions of Ma\~n\'e potential and Peierls barrier at any level $ \gamma $.
\begin{definition}
Let $ A : \Sigma \to \mathbb{R} $ be a potential and $ \gamma \in \mathbb{R} $ be a constant.
\begin{enumerate}
\item\label{defnition-mane-potential} We define the Ma\~n\'e potential as the function $ \phi_{A}^{\gamma} \,:\, \Sigma \times \Sigma \to \mathbb{R} \cup \{ \pm \infty \} $ given as 
\[
\phi_{A}^{\gamma} (x,y) 
:= \lim_{k \to \infty} \,
{}^{\gamma}\mathfrak{S}_{1}^{k} (x,y)
= \sup_{k \geq 0} \,
\inf_{n \geq 1} 
\inf_{ \stacksub{ d(x,w) \leq \lambda^k \\ d(\sigma^n(w),y) \leq \lambda^k } }
S_n ( \gamma - A )(w)
\]
for all $x$, $y \in \Sigma$.
\item\label{defnition-peierls-barrier} The Peierls barrier is the function $ h_{A}^{\gamma} \,:\, \Sigma \times \Sigma \to \mathbb{R} \cup \{ \pm \infty \} $ defined as
\begin{eqnarray*}
h_{A}^{\gamma} (x,y)
& := \sup_{k \geq 0} \,
\sup_{l \geq 1} \, {}^{\gamma}\mathfrak{S}_{l}^{k} (x,y) \\
& \, = \lim_{k \to \infty} \,
\liminf_{n \to \infty}
\inf_{ \stacksub{ d(x,w) \leq \lambda^k \\ d( \sigma^n(w),y ) \leq \lambda^k } }
S_n( \gamma - A)(w)
\end{eqnarray*}
for every $x$, $y \in \Sigma$. 
\end{enumerate}
\end{definition}

It is immediate from these definitions that 
\begin{equation}   \label{equation-mane-peierls-inequality}
\phi_{A}^{\gamma} (x,y) \leq h_{A}^{\gamma} (x,y)
\end{equation}
for all~$ x, y \in \Sigma $.
It is also easy to see that
$ {}^{\gamma}\mathfrak{S}_{l}^{k} \left( x, \sigma^n (x) \right) \leq S_n ( \gamma - A ) (x) $, 
for all $ k \geq 0 $ and $ n \geq l $, from which we obtain a fundamental inequality over an orbit
\begin{equation}   \label{equation-mane-orbit-inequality}
\phi^{\gamma}_{A} \left( x, \sigma^n ( x ) \right) 
\leq S_n ( \gamma - A )(x),
\end{equation}
for every~$ x \in \Sigma $ and for all~$ n \geq 1 $.

We can present basic ``triangle inequalities'' involving the Ma\~n\'e potential and the Peierls barrier.
\begin{proposition} \label{proposition-mane-peierls-triangular-inequalities}
Let $ A : \Sigma \to \mathbb{R} $ be a potential of summable variation and $ \gamma \in \mathbb{R} $ be a constant.
Then, for every point~$ x, y, z \in \Sigma $, the following inequalities hold
\begin{eqnarray}
\phi_{A}^{\gamma} (x,z) & \leq \phi_{A}^{\gamma} (x,y) + \phi_{A}^{\gamma} (y,z),
\label{equation-mane-triangle-inequality} \\[0.15cm]
h_{A}^{\gamma} (x,z) & \leq \phi_{A}^{\gamma} (x,y) + h_{A}^{\gamma} (y,z),
\label{equation-mane-peierls-triangle-inequality-1} \\[0.15cm]
h_{A}^{\gamma} (x,z) & \leq h_{A}^{\gamma} (x,y) + \phi_{A}^{\gamma} (y,z),
\label{equation-mane-peierls-triangle-inequality-2} \\[0.15cm]
h_{A}^{\gamma} (x,z) & \leq h_{A}^{\gamma} (x,y) + h_{A}^{\gamma} (y,z).
\label{equation-peierls-triangle-inequality} 
\end{eqnarray}
\end{proposition}

For completeness, we provide the proof following the approach of \cite[Proposition~5.2, Item~ii]{Gar:SPR17}.
\begin{proof}
It is straightforward from property
\eref{equation-auxiliary-monotonicity-l} and
Lemma~\ref{lemma-auxiliary-inequality} (applied with $ k = \bar{k} > 0 $)
that
$ {}^{\gamma}\mathfrak{S}_{l}^{ k } (x,z)
\leq {}^{\gamma}\mathfrak{S}_{l}^{ k } (x,y) 
+ {}^{\gamma}\mathfrak{S}_{m}^{ k } (y,z) 
+ 2\, \mathrm{Var}_{ k }^{\infty} (A) $.
Thus, \eref{equation-mane-triangle-inequality} follows by setting $ l = 1 $, $ m = 1 $ and taking the limit as $ k \to \infty $.
By setting $ m = 1 $, taking the supremum over $ l \geq 1 $ and the limit as $ k \to \infty $, we obtain~\eref{equation-mane-peierls-triangle-inequality-2}.
Similarly,
$ {}^{\gamma}\mathfrak{S}_{m}^{ k } (x,z) 
\leq {}^{\gamma}\mathfrak{S}_{l}^{ k } (x,y) 
+ {}^{\gamma}\mathfrak{S}_{m}^{ k } (y,z) 
+ 2\, \mathrm{Var}_{ k }^{\infty} (A) $, and inequality~\eref{equation-mane-peierls-triangle-inequality-1} follows
by setting $ l = 1 $, taking the supremum over $ m \geq 1 $ and then the limit as $ k \to \infty $.
Finally, \eref{equation-peierls-triangle-inequality} follows immediately from~\eref{equation-mane-peierls-triangle-inequality-1} and~\eref{equation-mane-peierls-inequality}.
\end{proof}

\subsection{Minus Infinity Dichotomy}
\label{subsection-minus-infinity-dichotomy}

We analyze a central dichotomy of the Ma\~n\'e potential (and the Peierls barrier) with respect to the value $ - \infty $, which is intrinsically related with the ergodic maximizing constant.

\begin{lemma}  \label{lemma-beta-critical}
Let $ A \,:\, \Sigma \to \mathbb{R} $ be a potential of summable variation.
Then
\begin{eqnarray*}
\beta_A
& =  \sup \left\{ \kappa \in \mathbb{R} \,:\,
\begin{array}{c}
\mbox{there exists a periodic point } x = \sigma^p (x) \\ \mbox{with } S_p ( \kappa - A )(x) < 0 
\end{array} \right\}  \\
& = \min \left\{ \kappa \in \mathbb{R} \,:\, 
\begin{array}{c}
S_p ( \kappa - A ) (x) \geq 0 \\ 
\mbox{ for all periodic point } x = \sigma^p (x) 
\end{array} \right\}.  
\end{eqnarray*} 
\end{lemma}
This corresponds to the Ma\~n\'e critical value characterization for $ \beta_A $ (see \cite{CI:IMPA99}).

\begin{proof}
Since
$ S_p ( \kappa - A )(x) < 0 \, \Leftrightarrow \, p \, \kappa < S_p A (x) $ and 
$ S_p ( \kappa - A ) (x) \geq 0 \, \Leftrightarrow \, S_p A (x) \leq p \, \kappa $,
the sets
\begin{eqnarray*}
\mathcal I & := \left\{ \kappa \in \mathbb{R} \,:\, \mbox{there is a periodic point } x = \sigma^p (x) \mbox{ with } p \, \kappa <  S_p A (x) \right\},  \\
\mathcal J & := \left\{ \kappa \in \mathbb{R} \,:\, S_p A (x) \leq p \, \kappa \mbox{ for all periodic point } x = \sigma^p (x) \right\}  
\end{eqnarray*} 
are complementary intervals with infinite endpoints such that $ \sup \, \mathcal I = \inf \, \mathcal I^{\complement} = \inf \, \mathcal J $. 
As $ \mathcal J $ is a closed set, the infimum is in fact a minimum.

Note that $ \beta_A \geq \min \mathcal J $.
In fact, given a periodic point $ x = \sigma^p (x) $, for the associated $ \sigma $-invariant probability 
$ \mu_x := \case{1}{p} \sum_{ i = 0 }^{ p - 1 } \delta_{ \sigma^i (x) } $, we obtain 
$ S_p A (x) = p \int A \, \rmd \mu_x \leq p \, \beta_A $.

From topological transitivity, periodic probabilities are dense among invariant measures, by Theorem~4.2 and Section~6 of \cite{CS:IJM10}, so that
\[
\beta_A 
= \sup \left\{ \case{1}{p} S_p A (x) \,:\,
x = \sigma^p (x) \mbox{ is a $p$-periodic point of } \Sigma
\right\}.
\]
It is easy to see that every $ \kappa \in \mathcal J $ is greater than or equal to $ \beta_A $, thus $ \beta_A \leq \min \mathcal J $.
\end{proof}

Now we can precisely state the fundamental dichotomy.
\begin{proposition}   \label{proposition-mane-minus-infinity-critical-value}
Let $ A \,:\, \Sigma \to \mathbb{R} $ be a potential of summable variation and $ \gamma \in \mathbb{R} $ be a constant.
The following assertions are equivalent:
\begin{enumerate}
\item\label{item-above-mane-critical-value} $ \gamma \geq \beta_A $;
\item\label{item-mane-minus-infinity-diagonal} $ \phi_{A}^{\gamma} \left( x, x \right) > - \infty $ for every $ x \in \Sigma $.
\end{enumerate} 
\end{proposition}
From the above proposition and the corresponding triangle inequality, the Ma\~n\'e potential (and the Peierls barrier) assumes the value $ - \infty $ everywhere or nowhere.

\begin{proof}
Let us prove the contrapositive statements.
Suppose that  $ \phi_{A}^{\gamma} \left( x, x \right) = - \infty $  for some $ x \in \Sigma $. It is immediate that
$ {}^{\gamma}\mathfrak{S}_{1}^{k} \left( x, x \right) = - \infty $ for any $ k $.
Since $ {}^{\gamma}\mathfrak{S}_{1}^{1} (x,x) $ is an infimum, there exist $ w \in \Sigma $ and integer $ n \geq 1 $ fulfilling
\[
d ( x, w ) \leq \lambda, \enspace\enspace
d ( \sigma^n(w), x ) \leq \lambda
\enspace\enspace \mbox{and} \enspace\enspace
S_n ( \gamma - A ) (w) < - \mathrm{Var}_{1}^{\infty} (A).
\]
In particular, $ w_0 = x_0 = w_{n} $, so we can consider the periodic point
\[
z := \big( 
x_0, w_{1}, \ldots, w_{n-1}, \;\;
x_0, w_{1}, \ldots, w_{n-1}, \;\;
\ldots \big) = \sigma^n (z) \in \Sigma.
\]
Note that
$ S_{ n }( \gamma - A )( z )
\leq S_{ n } ( \gamma - A )( w ) 
+ \mathrm{Var}_{1}^{\infty} (A)
< 0 $.
From Lemma~\ref{lemma-beta-critical}, $ \gamma < \beta_A $.

Assume now $ \gamma < \beta_A $.
Again Lemma~\ref{lemma-beta-critical} ensures that there is a periodic point $ x = \sigma^p (x) $ such that 
$- M := S_p ( \gamma - A )(x) < 0 $. 
By the periodicity of $ x $, from inequality~\eref{equation-mane-orbit-inequality} we conclude that
\begin{eqnarray*}
\phi_{A, \gamma} (x,x)
& = \lim_{k \to \infty} \phi_{A, \gamma} \left( x, \sigma^{ k \, p }(x) \right) \\
& \leq \lim_{k \to \infty} S_{ k \, p } ( \gamma - A )(x) 
=\lim_{k \to \infty} - k \, M
= -\infty.
\hspace*{3.27cm}\mbox{\qedsymbol}
\end{eqnarray*}
\let\qed\relax
\end{proof}

\begin{corollary}  \label{corollary-mane-peierls-minus-infinity} 
Let $ A \,:\, \Sigma \to \mathbb{R} $ be a potential of summable variation.
Then, for every $ x, y \in \Sigma $,
\[
- \infty < \phi_A \left( x, y \right)
\qquad\mbox{and}\qquad
-\infty < h_A \left( x, y \right).
\]
\end{corollary}

\begin{proof}
By Proposition~\ref{proposition-mane-minus-infinity-critical-value}, we have $ \phi_A \left( z, z  \right) > - \infty $.
Apply twice inequality~\eref{equation-mane-triangle-inequality} in order to obtain 
$ - \infty < \phi_A \left( z, z \right)
\leq \phi_A \left( z, x  \right) + \phi_A \left( x, y \right) + \phi_A \left( y, z \right) $ 
for any $ x $, $ y \in \Sigma $.
Thus, inequality~\eref{equation-mane-peierls-inequality} provides
$ - \infty < \phi_A \left( x, y \right) \leq h_A \left( x, y \right) $.
\end{proof}

\hypertarget{sec-aubry-set}{}
\section{Aubry Set}
\label{section-aubry-set}

The Aubry set was already introduced in Definition~\ref{definition-aubry-set}.
We provide below a list of the main properties of this set that remain unchanged regardless of the non-compact scenario. 
The proofs are presented below, adapting the arguments from~\cite[Chapter~4]{Gar:SPR17}.

\begin{proposition}   \label{proposition-aubry-properties}
Let $ A \,:\, \Sigma \to \mathbb{R} $ be a potential. 
The following properties hold.
\begin{enumerate}
\item\label{item-aubry-invariant} 
$ \Omega (A) $ is an invariant set, i.e., $ \sigma \left( \Omega (A) \right)  \subset \Omega (A) $.
\item\label{item-aubry-closed}
$ \Omega (A) $ is a closed set.
\item\label{item-aubry-maximizing-set}
If $ \mu $ is an $A$-maximizing measure, then $ \mathrm{supp} \, \mu \subset \Omega (A) $. 
In particular, the existence of a maximizing probability implies the Aubry set is non-empty.
\end{enumerate}
\end{proposition}

\begin{proof}
To prove item~\ref{item-aubry-invariant}, we decompose it into two steps:
$ \sigma \left( \Omega (A) \right) = \sigma \left( \Omega \left( A \circ \sigma \right) \right) \subset \Omega (A) $.
The equality 
$ \Omega (A) = \Omega \left( A \circ \sigma \right) $
follows from
$ \beta_{A} = \beta_{A \circ \sigma } $ and 
the fact that
$ \left| S_n ( \beta_A - A) - S_n \left( \beta_{A \circ \sigma} - A \circ \sigma \right) \right| = \left| A \circ \sigma^n - A \right| $ is sufficiently small for any recurrent point, due to the continuity of $ A $.
Furthermore, if $ x \in \Omega \left( A \circ \sigma \right) $, then, given $ k > 0 $, there exist a point $ w \in \Sigma $ and an integer $ n > 0 $ such that 
$ d (x, w) < \lambda^k $,
$ d (\sigma^{n} (w), x) < \lambda^k $
and 
$ - \lambda^k 
< S_n \left( \beta_{ A \circ \sigma } - A \circ \sigma \right) (w) 
= S_n \left( \beta_A - A \right) \left( \sigma(w) \right) 
< \lambda^k $.
Since
$ d \left( \sigma(x), \sigma(w) \right) < \lambda^{k-1} $ and
$ d \left( \sigma^n(\sigma(w)), \sigma(x) \right) < \lambda^{k-1} $,
we have that $ \sigma (x) \in \Omega (A) $ and thus $ \sigma \left( \Omega \left( A \circ \sigma \right) \right) \subset \Omega (A) $.

For item~\ref{item-aubry-closed}, let $ x $ be a limit point of some convergent sequence in $ \Omega (A) $. 
Given $ \varepsilon > 0 $, there exist an Aubry point $ \bar{x} $ in that sequence, satisfying $ d(x, \bar{x}) \leq \frac{\varepsilon}{2} $, a point $ w \in \Sigma $ and an integer $ n > 0 $ such that 
$ d(\bar{x}, w) < \frac{\varepsilon}{2} $,
$ d(\sigma^n(w), \bar{x}) < \frac{\varepsilon}{2} $
and 
$ - \frac{\varepsilon}{2} < S_n( \beta_A - A )(w) < \frac{\varepsilon}{2} $.
By the triangular inequality, $ w $ fulfills the condition in Definition~\ref{definition-aubry-set} with respect to $ x $ and $ \varepsilon $.
Therefore, $ x \in \Omega (A) $ and $ \Omega (A) $ is closed.

Finally, item~\ref{item-aubry-maximizing-set} follows as a direct application of Atkinson's theorem~\cite{Atk:JLMS76}.
\end{proof}

The behavior of the Ma\~n\'e potential and the Peierls barrier on the diagonal and the Aubry set are intimately related.
The following alternative characterization of Aubry points, given by Corollary~4.5 of~\cite{Gar:SPR17}, allows us to be more precise.

\begin{lemma}   \label{lemma-aubry-recurrent-definition}
Let $ A \,:\, \Sigma \to \mathbb{R} $ be a potential.
Then, $ x \in \Omega (A) $ if and only if
for any $ \varepsilon > 0 $ and for all integer $ L \geq 1 $, there are a point $ w \in \Sigma $ and an integer $ n \geq L $ such that 
$ d ( x, w ) < \varepsilon $, 
$ d ( \sigma^n (w), x ) < \varepsilon $
and
\[
- \varepsilon \leq S_n ( \beta_A - A)(w) \leq \varepsilon.
\]
\end{lemma}

\begin{proof}
Let $ x \in \Omega (A) $.
It is enough to argue that if $x$ admits a bounded family of positive integers $ \{ n ( \varepsilon ) \}_{ \varepsilon > 0 } $, each one satisfying the conditions in the Definition~\ref{definition-aubry-set} with respect to an associated point $ w ( \varepsilon ) $, then $ x $ is periodic and the ergodic $ \sigma $-invariant probability supported on its orbit is an $A$-maximizing measure.
In fact, we have 
$ d \left( x,  w ( \varepsilon ) \right) < \varepsilon $,
$ d \left( \sigma^{ n ( \varepsilon ) } \left( w ( \varepsilon ) \right), x \right) < \varepsilon $
and 
$ - \varepsilon < S_{ n ( \varepsilon ) } ( \beta_A - A ) \left( w ( \varepsilon ) \right) < \varepsilon $,
for each $ \varepsilon > 0 $.
It is straightforward to obtain a constant subsequence of integers
$ n (\varepsilon_k) = N $, for every integer $ k \geq 1 $,
with $ \varepsilon_k \to 0 $ as $ k \to \infty $. 
In particular, by the triangle inequality,
$ d \left( \sigma^N (x), x \right) 
< \left( 1 + \lambda^{-N} \right) \, \varepsilon_k $,
for all $ k $ sufficiently large, which implies that $ \sigma^N (x) = x $.
Moreover, since
$ - \varepsilon_k 
< S_{ N } ( \beta_A - A ) \left( w ( \varepsilon_k ) \right) 
< \varepsilon_k $, taking the limit as $ k \to \infty $,
it follows that
$ S_{N} A (x)  = N \, \beta_A $.
Due to the periodicity of $ x $, we conclude that 
$ S_{ l N } A (x) = \sum_{i = 0}^{l-1} S_{ N } A \left( \sigma^{ i N } (x) \right) = l \, N \, \beta_A $ for all~$ l \geq 1 $, which allows us to  consider~$ l \, N \geq L $.
\end{proof}

\begin{proposition}   \label{proposition-aubry-mane-peierls-zero}
Let $ A \,:\, \Sigma \to \mathbb{R} $ be a potential of summable variation. Then
\[
x \in \Omega (A)
\qquad\Longleftrightarrow\qquad
\phi_A ( x, x ) = h_A ( x, x ) = 0.
\]
\end{proposition}

\begin{proof}
Note first that $ 0 \leq \phi_A ( x, x ) \leq h_A ( x, x ) \leq +\infty $ for any $ x \in \Sigma $.
As a matter of fact,
$ \phi_A \left( x, \sigma (x) \right) \in \mathbb{R} $
by inequality~\eref{equation-mane-orbit-inequality} and Corollary~\ref{corollary-mane-peierls-minus-infinity}.
From inequalities~\eref{equation-mane-triangle-inequality} and~\eref{equation-mane-peierls-inequality}, it follows that 
$ 0 = \phi_A \left( x, \sigma (x) \right) - \phi_A \left( x, \sigma (x) \right) \leq \phi_A ( x, x ) \leq h_A ( x, x ) $ for all $ x \in \Sigma $.

Let $ x \in \Omega (A) $. 
From Lemma~\ref{lemma-aubry-recurrent-definition},  for every $ k \geq 0 $ and for all integer $ L \geq 1 $, there are  $ w \in \Sigma $ and $ n \geq L $ such that 
$ d ( x, w ) < \lambda^k $,
$ d ( \sigma^n(w), x ) < \lambda^k  $ and
\[
\mathfrak{S}_{L}^{k} (x, x) \leq S_n ( \beta_A - A ) (w) \leq \lambda^k.
\]
By taking the supremum with respect to $ L \geq 1 $ and passing to the limit as $ k \to \infty $, we obtain $ h_A ( x, x ) \leq 0 $.

Reciprocally, suppose $ h_A(x,x) = 0 $. In particular, $ \phi_A ( x, x ) = \sup_{k \geq 0} \, \mathfrak{S}_{1}^{k} ( x, x ) = 0 $.
Thus, given $ \varepsilon > 0 $, there is $ K \geq 0  $ such that 
$ - \varepsilon < \mathfrak{S}_{1}^{k} ( x, x ) \leq 0 $
for any $ k \geq K $. We may assume that $ \lambda^K \leq \varepsilon $.
For a fixed $ k \ge K $, since $ \mathfrak{S}_{1}^{k} ( x, x ) $ is an infimum, 
there are a point $ w \in \Sigma $ and an integer $ n \geq 1 $ fulfilling
$ d ( x, w ) \leq \lambda^k \leq  \varepsilon $, \,
$ d ( \sigma^n (w), x ) \leq \lambda^k \leq  \varepsilon $ and
\[
- \varepsilon 
< \mathfrak{S}_{1}^{k} ( x, x ) 
\leq S_n ( \beta_A - A ) (w) 
< \mathfrak{S}_{1}^{k} ( x, x ) + \varepsilon 
\leq \varepsilon.
\]
Therefore, $ x $ is an Aubry point.
\end{proof}

\subsection{Sub-action}
\label{subsection-sub-action}

We will show that the existence of an Aubry point $ x $ ensures that there is always a continuous sub-action, precisely the function
\[
y \in \Sigma  \,\, \longmapsto \,\, u_x(y):= \phi_A(x, y) = h_A(x, y) \in \mathbb R.
\]
The first step is to observe that we are dealing with a real-valued function.

\begin{proposition}   \label{proposition-mane-peierls-real-valued} 
Let $ A \,:\, \Sigma \to \mathbb{R} $ be a potential of summable variation.
If $ h_A (x, z) \in \mathbb{R} $ for some $ x, z \in \Sigma $, then 
\[
\phi_A \left( x, y \right) \in \mathbb{R} 
\qquad \mbox{and} \qquad
h_A \left( x, y \right) \in \mathbb{R}, \qquad \forall \, y \in \Sigma.
\]
\end{proposition} 

\begin{proof}
Property~\eref{equation-auxiliary-monotonicity-l} and Lemma~\ref{lemma-auxiliary-inequality} (with $ m = 0 $) provide for any $ x, y, z \in \Sigma $,
\[
\sup_{ l \geq 1 } \, \mathfrak{S}_{l}^{k} \left( x, y \right) \leq
\sup_{ l \geq k - \bar{k} } \, \mathfrak{S}_{l}^{k} \left( x, y \right)  
\leq \sup_{ l \geq 1 } \, \mathfrak{S}_{l}^{k} \left( x, z \right)
+ \mathfrak{S}_{0}^{\bar{k}} \left( z, y \right) 
+ 2 \, \mathrm{Var}^{\infty}_{1} (A),
\]
where $ k \geq \bar{k} > 0 $.
By passing to the limit as $ k \to \infty $, and recalling inequality~\eref{equation-mane-peierls-inequality}, we obtain for any $ \bar{k} > 0 $,
\[
\phi_A \left( x, y \right) 
\leq h_A \left( x, y \right) 
\leq h_A \left( x, z \right) + \mathfrak{S}_{0}^{\bar{k}} \left( z, y \right) + 2 \, \mathrm{Var}^{\infty}_{1} (A).
\]
We conclude the result applying Corollary~\ref{corollary-mane-peierls-minus-infinity}, the hypothesis and~\eref{equation-auxiliary-non-plus-infinity}.
\end{proof}

\begin{corollary}   \label{corollary-aubry-mane-peierls-real-valued}
Let $ A \,:\, \Sigma \to \mathbb{R} $ be a potential of summable variation.
If $ x \in \Omega (A) $, then 
\[
\phi_A (x, y) = h_A (x, y) \in \mathbb{R}, \qquad \forall \, y \in \Sigma.
\]
\end{corollary}

\begin{proof}
The equality between the Ma\~n\'e potential and the Peierls barrier in this case follows from inequalities~\eref{equation-mane-peierls-inequality} and \eref{equation-mane-peierls-triangle-inequality-2}, and from Proposition~\ref{proposition-aubry-mane-peierls-zero}, since
\[
\phi_A (x,y) 
\leq h_A (x,y) 
\leq h_A (x,x) + \phi_A (x,y) 
= \phi_A (x,y). 
\]
Besides, $ h_A (x, y) \in \mathbb{R} $ thanks to Propositions~\ref{proposition-aubry-mane-peierls-zero} and~\ref{proposition-mane-peierls-real-valued}.
\end{proof}

The Peierls barrier is continuous with respect to the second variable.

\begin{proposition}   \label{proposition-peierls-summable-variation}
Let $ A \,:\, \Sigma \to \mathbb{R} $ be a potential of summable variation.
If $ h_A ( \bar{x}, z ) \in \mathbb{R} $ for some $ \bar{x}, z \in \Sigma $, then the map
\[
y \in \Sigma \,\, \longmapsto \,\, h_A (\bar{x}, y) \in \mathbb{R}
\]
is continuous with $\ell$-th variation bounded from above by $ 2 \, \mathrm{Var}_{\ell}^\infty (A) $.
In particular, if $ \sum_{\ell = 1}^\infty \mathrm{Var}_{\ell}^\infty (A) < + \infty $, then $  h_A (\bar{x}, \cdot) $  is of summable variation.
\end{proposition}

\begin{proof}
Proposition~\ref{proposition-mane-peierls-real-valued} guarantee that $ h_A (\bar{x}, \,\cdot\, ) $ is a real-valued function.
Consider points $ x, y \in \Sigma $ such that $ d ( x, y ) \le \lambda^{ \bar{k} } $, with $ \bar{k} > 0 $.
From property~\eref{equation-auxiliary-monotonicity-l} and Lemma~\ref{lemma-auxiliary-inequality} (with $ m = 0 $), we obtain 
\[
\sup_{ l \geq 1 } \, \mathfrak{S}_{l}^{k} \left( \bar{x}, x \right) \leq
\sup_{ l \geq k - \bar{k} } \, \mathfrak{S}_{l}^{k} \left( \bar{x}, x \right)  
\leq \sup_{ l \geq 1 } \, \mathfrak{S}_{l}^{k} \left( \bar{x}, y \right)
+ \mathfrak{S}_{0}^{\bar{k}} \left( y, x \right) 
+ 2 \, \mathrm{Var}^{\infty}_{\bar{k}} (A),
\]
where $ k \geq \bar{k} > 0 $.
Note that $ \mathfrak{S}_{0}^{\bar{k}} ( y, x ) \leq S_0 ( \beta_A - A ) ( x ) = 0 $. 
As $ k \to \infty $, it follows that
$ h_A \left( \bar{x}, x \right) - h_A \left( \bar{x}, y \right) 
\leq 2 \, \mathrm{Var}^{\infty}_{\bar{k}} (A) $,
for any $ x, y \in \Sigma $ with $ d ( x, y ) \le \lambda^{ \bar{k} } $.
In other terms,
\[ 
\mathrm{Var}_{\ell} \left( h_A (\bar{x}, \, \cdot \, )  \right)
\leq 2 \,  \mathrm{Var}_{\ell}^\infty (A).
\hspace*{7.9cm}\mbox{\qedsymbol}
\]
\let\qed\relax
\end{proof}

\begin{remark}
There is a slight loss of regularity between the potential $ A $ and the associated Peierls barrier, which seems to be natural on non-compact scenarios, see~\cite{GI:LMP22}.    
\end{remark}

Regularity on the first coordinate can be verified for Aubry points.

\begin{proposition}   \label{proposition-mane-peierls-summable-variation-aubry}
Let $ A \,:\, \Sigma \to \mathbb{R} $ be a potential of summable variation.
Then, for any $ \bar{y} \in \Sigma $ fixed, the map 
\[
x \in \Omega (A) \,\, \longmapsto \,\, \phi_A ( x, \bar{y} ) = h_A ( x, \bar{y} ) \in \mathbb{R}
\]
is continuous with $\ell$-th variation bounded from above by $ 2 \, \mathrm{Var}_{\ell}^\infty (A) $.
In particular, if  
$ \sum_{\ell = 1}^\infty \mathrm{Var}_{\ell}^\infty (A) < + \infty $, then $  h_A(\cdot , \bar{y}) |_{ \Omega (A)} $ is  of summable variation.
\end{proposition}

\begin{proof}
Thanks to Corollary~\ref{corollary-aubry-mane-peierls-real-valued},
$ \phi_A ( x, \,\cdot\, ) = h_A ( x, \,\cdot\, ) \in \mathbb{R} $,
for every $ x \in \Omega (A) $.
By inequality~\eref{equation-peierls-triangle-inequality} and Proposition~\ref{proposition-aubry-mane-peierls-zero}, we obtain
\[
h_A (x, \bar{y}) - h_A (y, \bar{y}) 
\leq h_A (x, y) = h_A (x, y) - h_A (x, x),
\]
for all $ x, y \in \Omega (A) $.
Hence, if $ d(x, y) \le \lambda^\ell $ with $ \ell \geq 1 $, Proposition~\ref{proposition-peierls-summable-variation} ensures that
\[ 
\mathrm{Var}_{\ell} \left( h_A ( \,\cdot\, , \bar{y}) |_{ \Omega (A) } \right) \leq 2 \, \mathrm{Var}_{\ell}^\infty (A).
\hspace*{7.05cm}\mbox{\qedsymbol}
\]
\let\qed\relax
\end{proof}

We can now ensure the existence of a basic sub-action, 
obtained from the Ma\~n\'e potential and the Peierls barrier
(also known as \hyperlink{CThm}{Collateral Theorem}).

\begin{theorem}   \label{theorem-sub-action}
Let $ A \,:\, \Sigma \to \mathbb{R} $ be a potential of summable variation.
Then, for any $ x \in \Omega (A) $ fixed, the map 
\[
\begin{array}{ r c l}
u_{x} ( \,\cdot\, ) = \phi_A ( x, \,\cdot\, ) = h_A( x, \,\cdot\, ) \,:\, \Sigma & \longrightarrow & \mathbb{R} \\

y  & \longmapsto & u_x(y) = \phi_A ( x ,y ) = h_A( x, y ) 
\end{array}
\]
is a continuous sub-action with  $\ell$-th variation bounded from above by $ 2 \, \mathrm{Var}_{\ell}^\infty (A) $. 
In particular, if $ \sum_{\ell = 1}^\infty \mathrm{Var}_{\ell}^\infty (A) < + \infty $, then $ u_x $ is a sub-action  of summable variation.
More specifically, if $ A \,:\, \Sigma \to \mathbb{R} $ is a locally H\"older continuous potential, then $ u_x $ is a locally H\"older continuous sub-action.
\end{theorem}

\begin{proof}
First, note that $ u_x $ fulfills the inequality in the definition of a sub-action. 
Indeed, by applying inequalities~\eref{equation-mane-triangle-inequality} and~\eref{equation-mane-orbit-inequality}, for every~$ y \in \Sigma $, we get
\begin{eqnarray*}
u_{x} \circ \sigma (y) = \phi_A \left( x, \sigma (y) \right) 
& \leq \phi_A \left( x, y \right) + \phi_A \left( y, \sigma (y) \right) \\
& \leq u_{x} ( y ) + ( \beta_A - A )( y )
= u_{x} ( y ) - A ( y ) + \beta_A.
\end{eqnarray*}
Next, the statements about the regularity of the function $ u_x $ are direct consequences of Corollary~\ref{corollary-aubry-mane-peierls-real-valued} and Proposition~\ref{proposition-peierls-summable-variation}. 
\end{proof}

When the potential $ A $ admits a sub-action $ u $, we introduce its contact locus as 
\[
\mathbb{M}_A \left( u \right) := ( A + u \circ \sigma - u )^{-1} \left( \beta_A \right).
\]
We summarize the main properties of this set including its relations with the Aubry set and the maximizing measures.
Proofs are standard adaptations from the compact case~\cite{Gar:SPR17} and are provided below for completeness. 

\begin{proposition}   \label{proposition-contact-properties}
Let $ A \,:\, \Sigma \to \mathbb{R} $ be a potential and $ u \,:\, \Sigma \to \mathbb{R} $ be any sub-action of~$ A $. 
The following properties hold.
\begin{enumerate}
\item\label{item-contact-closed} $ \mathbb{M}_A \left( u \right) $ is a closed set.
\item\label{item-aubry-subset-contact} 
$ \Omega (A) \subset \mathbb{M}_A \left( u \right) $.
\item\label{item-contact-maximizing-set} If
$ \mu $ is an $A$-maximizing measure, then $ \mathrm{supp} \, \mu \subset \mathbb{M}_A \left( u \right) $.
In particular, $ \mathbb{M}_A \left( u \right) $ is a non-empty set whenever there exists a maximizing probability.
\end{enumerate}
\end{proposition}

\begin{proof}
Item~\ref{item-contact-closed} is immediate, since $ A $ and $ u $ are continuous.
To prove item~\ref{item-aubry-subset-contact}, let $ x \in \Omega (A) $. 
For each $ k > 0 $, there exist a point $ w^k \in \Sigma $ and an integer $ n_k > 0 $ such that 
$ d(x, w^k) < \lambda^{k} $,
$ d(\sigma^{n_k} (w^k), x) < \lambda^{k} $
and 
$ - \lambda^{k} < S_{n_k} ( \beta_A - A )(w^k) < \lambda^{k} $.
In particular, we have the following convergences as $ k \to \infty $:
\[
\fl
w^k \to x,
\quad
\sigma^{n_k} (w^k) \to x,
\quad
S_{n_k} ( \beta_A - A )(w^k) \to 0
\qquad 
\mbox{and} 
\qquad
u \circ \sigma^{n_k} (w^k) - u (w^k) \to 0.
\]
Using the definition of a sub-action, note that 
$ 0 \leq \beta_A - ( A - u \circ \sigma + u ) (w^k)
\leq S_{n_k} \left( \beta_A - A + u \circ \sigma - u \right) (w^k) 
= S_{n_k} ( \beta_A - A ) (w^k) + u \circ \sigma^{n_k} (w^k) - u (w^k) $
which converges to zero as $ k \to \infty $.
Hence, $ A - u \circ \sigma + u ) (x) = \beta_A $ and $ x \in \mathbb{M}_A \left( u \right) $.
For item~\ref{item-contact-maximizing-set}, the result follows from item~\ref{item-aubry-maximizing-set} of Proposition~\ref{proposition-aubry-properties} and item~\ref{item-aubry-subset-contact} above. 
\end{proof}

\hypertarget{sec-densely-periodic-optimization}{}
\section{Densely Periodic Optimization}
\label{section-densely-periodic-optimization}

This section is dedicated to prove the \hyperlink{MThm}{Main Theorem}.
The argument is inspired by the proof of Contreras theorem for shifts over finite alphabets \cite{Con:IM16,HLMXZ:JEMS25,Boc:Notes19}.
In particular, we will make use of the following result.

\begin{lemma}[Huang, Lian, Ma, Xu, Zhang]  \label{lemma-HLMXZ}
Let $ \Omega $ be a compact invariant subset of a Markov shift over a finite alphabet.
Then, for any $ \tau > 0 $, there exists a periodic orbit  $ \mathcal O $ such that
\[
\sum_{z \in \mathcal O} d(\Omega, z) < \tau \Delta(\mathcal O), 
\]
where $ \Delta(\mathcal O) $ denotes the half-gap of the orbit: 
$ \Delta(\mathcal O) = \frac{1}{2} \min \, \Big\{ \lambda, \min_{ \stacksub{ y,z \in \mathcal O \\ y \neq z} } d(y,z) \Big\} $\footnote{It is conventional that $ \min \emptyset = \infty $, so that $ \Delta(\mathcal O) = \lambda / 2 $ if $ \mathcal O $ consists of a single fixed point.}.
\end{lemma}

For a proof of this lemma, see \cite[Proposition~2.1]{HLMXZ:JEMS25} or \cite[Lemma~2.3]{Boc:Notes19}.

\begin{proof}[Proof of the \hyperlink{MThm}{Main Theorem}]
Let $ \Omega $ denote the non-empty compact invariant subset of the Aubry set of the potential $ A $ (that can be viewed as being contained in a Markov subshift over a finite alphabet).
For $ x \in \Omega $, we consider the associated locally H\"older continuous sub-action $ u := \phi_A(x, \cdot) = h_A(x, \cdot) $ given by Theorem~\ref{theorem-sub-action}, and we introduce
\[
\hat A := \frac{1}{1 + \mathrm{Lip}_{\mathrm{loc}} \left( A + u\circ\sigma - u \right) } \, \big(A + u\circ\sigma - u - \beta_A \big). 
\]
Note that $\hat A \le 0 $ is a locally H\"older continuous potential, with local H\"older constant less than~$1$, which admits among its maximizing probabilities all $A$-maximizing measures with compact support.

Since maximizing measures are preserved when multiplying a potential by a positive constant, it suffices to show that, given $ \varepsilon > 0 $, there exists a periodic point $ y = y(\varepsilon) \in \Sigma $ such
that, if the perturbed potential 
\[
\hat A - \varepsilon d(\mathrm{orb}(y), \cdot) 
\]
has a maximizing probability, then this is necessarily the periodic measure
\[
\mu_y : = \frac{1 }{\# \mathrm{orb}(y)} \sum_{z \in \mathrm{orb}(y)} \delta_z. 
\]

By taking $ \tau \in (0, \varepsilon) $, we apply Lemma~\ref{lemma-HLMXZ}
to determine a periodic point $ y $ satisfying 
\[
\sum_{z \in \mathrm{orb}(y)} d(\Omega, z) < \tau \Delta,
\]
where we abbreviate  $ \Delta = \Delta( \mathrm{orb}(y) ) $. For concreteness, we choose once and for all
\[
 \tau := \frac{1}{4} \cdot \frac{\varepsilon}{1+ \frac{1}{1-\lambda}\cdot\frac{1}{\varepsilon}}. 
\]
We now denote
\[
\beta(\varepsilon) := \beta_{\hat A - \varepsilon d(\mathrm{orb}(y), \cdot)} \le 0 
\]
and define
\[
\hat B := \hat A - \varepsilon d(\mathrm{orb}(y), \cdot) - \beta(\varepsilon). 
\]
It therefore suffices to show that if there exists a $\hat B$-maximizing probability, then it must necessarily be $ \mu_y $.

We have $  \int \hat A \, d\mu_y = \int [\hat A - \varepsilon d(\mathrm{orb}(y), \cdot)] \, d\mu_y \le \beta(\varepsilon) $. 
In particular, we note that
\begin{equation}  \label{equation-sinal}
\displaystyle \int \hat B \, d\mu_y \le 0.
\end{equation}
Moreover, since $ \mathrm{Lip}_{\mathrm{loc}} \big( \hat A \big) \leq 1 $ and $ \hat A |_\Omega \equiv 0 $ (by item~\ref{item-aubry-subset-contact} of Proposition~\ref{proposition-contact-properties}),
\[
\int \hat A \, d\mu_y
= \frac{1}{\# \mathrm{orb}(y)} \sum_{z \in \mathrm{orb}(y)} \hat A(z)
\ge - \frac{1}{\# \mathrm{orb}(y)} \sum_{z \in \mathrm{orb}(y)} d(\Omega, z), 
\]
which yields 
\[
\int \hat A \, d\mu_y \ge - \frac{1}{\# \mathrm{orb}(y)} \tau \Delta.
\]
We then obtain that 
\begin{equation}  \label{equation-majoracao-perturbacao-constante}
 \hat B \le - \beta(\varepsilon) \le \frac{1}{\# \mathrm{orb}(y)} \tau \Delta.
\end{equation}
Let $ r := - \beta(\varepsilon) / \varepsilon $. We observe that
\begin{equation}  \label{equation-afastamento-sinal}
 d(\mathrm{orb}(y), z) > r \qquad \Rightarrow \qquad \hat B(z) \le 0. 
\end{equation}
Indeed, in this case,
\[
\hat B(z) = \hat A(z) - \varepsilon d(\mathrm{orb}(y),z) - \beta(\varepsilon) \le \beta(\varepsilon) - \beta(\varepsilon) = 0. 
\]

Now suppose that $ \hat B $ admits a maximizing probability. 
From item~\ref{item-aubry-maximizing-set} of Proposition~\ref{proposition-aubry-properties}, we have $ \Omega (\hat B) \neq \emptyset $. Moreover, it suffices to argue that $ \Omega (\hat B) = \mathrm{orb}(y) $ to conclude that $ \mu_y $ is the only possible maximizing measure for $ \hat B $.

Let $ \hat x \in \Omega (\hat B) $. Given $ \hat \varepsilon \in (0, \varepsilon \Delta / 2) $, for every integer $ L \ge 1 $, by Lemma~\ref{lemma-aubry-recurrent-definition}, there exist a point $ w_L \in \Sigma $ and an integer $ N_L \ge L $ such that
$ d ( \hat{x}, w_L ) < \hat{\varepsilon} $,
$ d ( \sigma^{N_L} ( w_L ), \hat{x} ) < \hat{\varepsilon} $ and 
\[
- \hat{\varepsilon} < S_{N_L} \hat{B} ( w_L ) < \hat{\varepsilon}.
\]
Let $ V_{\Delta} \left( \mathrm{orb}(y) \right) $ denote the set of points in $ \Sigma $ that are at distance at most $ \Delta $ from the periodic orbit of $y$. It suffices to show that $ \sigma^\ell(w_L) \in V_{\Delta} \left( \mathrm{orb}(y) \right) $ for all $ 0 \le \ell < N_L $. In fact, if this is the case, then by the definition of $ \Delta $ and by the expansiveness, there exists an integer $ \hat k = \hat k (\hat \varepsilon, L) \ge 0 $ such that 
$ d ( w_L, \sigma^{\hat k}(y) ) \le \lambda^{N_L} \Delta $, which implies  $ d ( \hat x, \sigma^{\hat k}(y) ) < \hat \varepsilon + \lambda^{N_L} \Delta $. Since $ N_L $ can be taken arbitrarily large and $ \hat \varepsilon $ can be taken arbitrarily small, we conclude that $ \hat x \in \mathrm{orb}(y) $. 

We decompose
\[
S_{N_L} \hat B (w_L) = \sum_{m=0}^M \sum_{\ell_m < \ell \le \ell_{m+1}} \hat B(\sigma^\ell(w_L)),
\]
where $ \ell_0 = 0 $ and $ \ell_{M+1} = N_L - 1 $ by convention, and $ 0 \le \ell_1 < \ldots < \ell_M < N_L $ are (if they exist) the successive indices for which $ d(\mathrm{orb}(y), \sigma^{\ell_m}(w_L)) > \Delta $. Our goal is to prove that the choice of $ \hat \varepsilon $ precludes the existence of such indices.

First, we observe that from~\eref{equation-majoracao-perturbacao-constante},
\[
r = \frac{-\beta(\varepsilon)}{\varepsilon} \le \frac{1}{\# \mathrm{orb}(y)} \frac{\tau}{\varepsilon} \Delta \le
 \frac{\tau}{\varepsilon} \Delta < \Delta. 
\]
Whenever the set 
\[
 \{ \ell: \ell_m < \ell < \ell_{m+1} \quad \mathrm{and} \quad
d(\mathrm{orb}(y), \sigma^\ell(w_L)) \le r \} 
\] 
is empty, using~\eref{equation-afastamento-sinal} when necessary, we note that
\[
\sum_{\ell_m < \ell < \ell_{m+1}} \hat B(\sigma^\ell(w_L)) \le 0. 
\]
Otherwise, we denote by $ k_m $ the maximum of this set, and we have
\[
\sum_{\ell_m < \ell < \ell_{m+1}} \hat B(\sigma^\ell(w_L)) \le
\sum_{\ell_m < \ell \le k_m} \hat B(\sigma^\ell(w_L)). 
\]
Since $ d(\mathrm{orb}(y), \sigma^{k_m}(w_L)) \le r < \Delta $, the point on the orbit of $ y $ that achieves this distance is uniquely determined. Denoting it by $ \sigma^{k_m}(z) $ for some $ z \in \mathrm{orb}(y) $, by the definition of $ \Delta $ and by the expansiveness, we obtain 
\[
d(\mathrm{orb}(y), \sigma^\ell(w_L)) = d(\sigma^\ell(z), \sigma^\ell(w_L)) \le \lambda^{k_m - \ell} d( \sigma^{k_m}(z),\sigma^{k_m}(w_L))
\]
for $ \ell \in \{\ell_m + 1, \ldots, k_m\} $.
Recalling that $ \mathrm{Lip}_{\mathrm{loc}} \big( \hat A \big) \leq 1 $, we can estimate
\begin{eqnarray*}
\sum_{\ell_m < \ell \le k_m}& \Big[\hat B(\sigma^\ell(w_L)) - \hat B(\sigma^\ell(z)) \Big] \le \\
& \le (1 + \lambda + \ldots + \lambda^{k_m - (\ell_m+1)}) \, r -
\varepsilon \sum_{\ell_m < \ell \le k_m} d(\mathrm{orb}(y), \sigma^\ell(w_L)) \\
& \le \frac{1}{1-\lambda} r. 
\end{eqnarray*}
Therefore, the following upper bound holds:
\begin{eqnarray*}
\sum_{\ell_m < \ell < \ell_{m+1}} \hat B(\sigma^\ell(w_L)) 
& \le \frac{1}{1-\lambda} r + \left\lfloor \frac{k_m - \ell_m}{\# \mathrm{orb}(y)} \right\rfloor
\# \mathrm{orb}(y) \int \hat B \, d\mu_y \, \\
& \quad
+ (\# \mathrm{orb}(y) - 1)  \frac{1}{\# \mathrm{orb}(y)} \tau \Delta \\
& \le \frac{1}{1-\lambda} r + \Big(1 - \frac{1}{\# \mathrm{orb}(y)} \Big) \tau \Delta \\ 
& \le \Big( \frac{1}{1-\lambda}\cdot\frac{1}{\varepsilon}\cdot\frac{1}{\# \mathrm{orb}(y)} + 1 - \frac{1}{\# \mathrm{orb}(y)}  \Big) \tau \Delta.
\end{eqnarray*}
(We use~\eref{equation-majoracao-perturbacao-constante} to establish the first and third inequalities, and we use only~\eref{equation-sinal} to obtain the second inequality.)
Clearly, for $ 0 \le m < M $,
\begin{eqnarray*}
\hat B(\sigma^{\ell_{m+1}}(w_L)) & = \hat A(\sigma^{\ell_{m+1}}(w_L)) - \varepsilon d(\mathrm{ orb}(y), \sigma^{\ell_{m+1}}(w_L)) - \beta(\varepsilon) \\
& \le - \varepsilon \Delta + \frac{1}{\# \mathrm{orb}(y)} \tau \Delta.
\end{eqnarray*}
Thus, whether or not $ k_m $ exists, for $ 0 \le m < M $, we deduce that
\begin{eqnarray*}
\sum_{\ell_m < \ell \le \ell_{m+1}} \hat B(\sigma^\ell(w_L))
& \le \Big[ \Big( \frac{1}{1-\lambda}\cdot\frac{1}{\varepsilon}\cdot\frac{1}{\# \mathrm{orb}(y)} + 1 \Big) \tau - \varepsilon \Big] \Delta \\
& \le \Big[ \Big( \frac{1}{1-\lambda}\cdot\frac{1}{\varepsilon} + 1 \Big) \tau - \varepsilon \Big] \Delta = - \frac{3\varepsilon}{4}\Delta.
\end{eqnarray*}
For $ \ell_M < N_L - 1 $, the previous bounds give us
\[
 \sum_{\ell_{M-1} < \ell \le \ell_{M+1}} \hat B(\sigma^\ell(w_L)) \le
 \Big[ 2\Big( \frac{1}{1-\lambda}\cdot\frac{1}{\varepsilon} + 1 \Big) \tau - \varepsilon \Big] \Delta = - \frac{\varepsilon}{2}\Delta.
\]
These estimates ensure that
\[
S_{N_L} \hat{B} ( w_L )
\leq - \frac{\varepsilon}{2} \Delta \, S_{N_L} \mathds{1}_{\Sigma \setminus V_{\Delta} \left( \mathrm{orb}(y) \right) } ( w_L ).
\]  
From the very definition of $ w_L $, it follows that 
\[
S_{N_L} \mathds{1}_{\Sigma \setminus V_{\Delta} \left( \mathrm{orb}(y) \right) } ( w_L )
\leq \hat{\varepsilon} \frac{2}{ \varepsilon \, \Delta } < 1.
\]
Thus, we have shown that $ \sigma^{\ell} ( w_L ) \in V_{\Delta} \left( \mathrm{orb}(y) \right) $ for all $ 0 \leq \ell < N_L $.
\end{proof}

\hypertarget{app-A}{} 
\section{Appendix A - The Finite Cyclic Predecessor Case} 
\label{appendix-A} 

In this appendix, we exhibit a family of countable Markov shifts for which a dense subclass of locally H\"older continuous potentials (including coercive ones) admits exactly one periodic probability as maximizing measure. This strengthened result addresses the distinction between ``at most'' and ``exactly one'' discussed in the main text, showing that under additional structural assumptions, periodic optimization can be guaranteed to occur.

The key structural assumption in this context is the finite cyclic predecessor assumption. Intuitively, this condition requires that the shift space contains very well-behaved periodic behavior: for any given period and any initial symbol, there must exist finitely many periodic orbits of that period which begin with that symbol. More precisely, for every positive integer $p$ and for every cylinder set $[i]$, there exist finitely many periodic orbits of period $p$ that intersect $[i]$. The result we invoke also requires that periodic probabilities are dense among invariant measures, which for transitive Markovian shifts is guaranteed by Theorem~4.2 and Section~6 of \cite{CS:IJM10}.

We apply the main result of \cite{GGM:ETDS26} to this setting. Although the original result is more comprehensive and covers other contexts (including non-Markovian scenarios and upper semi-continuous potentials), we present an adapted version tailored to our setting.

\begin{theorem-E} \hypertarget{EThm}{} 
Let $(\Sigma, \sigma)$ be a topologically transitive Markov shift on a countable alphabet that satisfies the finite cyclic predecessor assumption. Then, every continuous potential $B$ fulfilling 
\[ 
\limsup_{i \to \infty} \sup B |_{[i]} < \beta_{B} 
\] 
has a maximizing probability. 
\end{theorem-E}

The central condition demands the potential's supremum on cylinder sets $[i]$ as $i \to \infty$ remains strictly below its ergodic maximizing constant. The key insight is that this requirement compels optimization to occur ``far from the infinite boundaries'' of the symbol space. The techniques in \cite{GGM:ETDS26} rely on the blur shift compactification method, which adds new symbols given by blurred subsets of the alphabet to achieve compactness while preserving the essential dynamics.

From the above result, we present the following strengthened version of our \hyperlink{MThm}{Main Theorem}, in which we can replace ``at most'' by ``exactly one'' in the conclusion.

\begin{theorem-F}  \hypertarget{FThm}{}
Let $ (\Sigma, \sigma) $ be a topologically transitive Markov shift on a countable alphabet that satisfies the finite cyclic predecessor assumption. Let $ A $ be a locally H\"older continuous potential. Suppose that its Aubry set contains a non-empty compact invariant subset. Then, given $ \varepsilon > 0 $, there exists a locally H\"older continuous potential $ B $ such that $ \| A - B \|_{\mathrm{sv}} < \varepsilon $ and $ B $ admits a single periodic probability as maximizing measure.
\end{theorem-F}

\begin{proof} We maintain all notation from the proof of our \hyperlink{MThm}{Main Theorem}, presented in the \hyperlink{sec-densely-periodic-optimization}{previous section}. In particular, we take into account the auxiliary potential $ \hat{A} $ as well as the periodic point  $ y $. The strategy is to show that the perturbed potential 
\[
B = \hat A - \varepsilon d(\mathrm{orb}(y), \cdot) 
\] 
satisfies the condition 
\[
\limsup_{i \to \infty} \sup B|_{[i]} < \beta(\varepsilon),
\] 
which, by the theorem above, guarantees the existence of a $B$-maximizing measure. 
Combined with the uniqueness established in our main theorem, this yields exactly one periodic maximizing measure.

Recall from the main proof that 
\[ 
\int B \, d\mu_y = \int \hat{A} \, d\mu_y \geq - \frac{1}{\# \mathrm{orb}(y)} \tau \Delta, 
\] 
where $\mu_y$ is the periodic measure supported on the orbit of $y$.
From the construction in the main theorem, we have $\tau \Delta < \# \mathrm{orb}(y) \, \varepsilon$. 
This allows us to estimate: 
\begin{eqnarray*}
\limsup_{i \to \infty} \, \sup B|_{[i]} 
& \leq \limsup_{i \to \infty} \, \sup \hat{A}|_{[i]}
- \varepsilon \, \liminf_{i \to \infty} \, \inf  d(\mathrm{orb}(y), \cdot)|_{[i]} 
\\
& \leq - \varepsilon 
< - \frac{1}{\# \mathrm{orb}(y)} \tau \Delta
\leq \beta ( \varepsilon ).
\end{eqnarray*}
The first inequality follows from the definition of $B = \hat{A} - \varepsilon d(\mathrm{orb}(y), \cdot)$ and standard estimates. The second inequality holds because the distance function $d(\mathrm{orb}(y), \cdot)|_{[i]} = 1 $ for  $ i $ large enough, since the periodic orbit is formed by finitely many symbols. The final inequality uses our estimate on $\tau \Delta$ and the definition of $\beta(\varepsilon)$.
Therefore, the perturbed potential $B$ satisfies the hypothesis of the~\hyperlink{EThm}{Existence Theorem}.
\end{proof}

\ack 

E G was supported by FAPESP project 2019/10.269-3 and ANR project IZES ANR-22-CE40-0011.


\end{document}